\documentclass{article}
\usepackage{amsmath}
\usepackage{amsthm}
\usepackage{mathtext}
\usepackage{amssymb}
\usepackage{tocenter}

\ToCenter{450pt}{650pt}

\newtheorem{thm}{Theorem}
\newtheorem*{thm*}{Theorem}
\newtheorem{cor}{Corollary}
\newtheorem{lem}{Lemma}
\newtheorem*{lem*}{Lemma}
\newtheorem{prop}{Proposition}
\newtheorem*{con*}{Conjecture}
\newtheorem*{prob*}{Problem}
\theoremstyle{definition}
\newtheorem{defn}{Definition}
\newtheorem*{ex*}{Example}
\newtheorem{cons}{Construction}
\newtheorem*{cons*}{Construction}
\newtheorem{rem}{Remark}
\theoremstyle{remark}
\newtheorem*{not*}{Notation}

\begin{document}

\title{Geometric realization of $\gamma$-vectors of 2-truncated cubes.}%
\author{V.~D.~Volodin\thanks{This work is supported by the Russian Government project 11.G34.31.0053.}}%
\date{}

\maketitle
\begin{abstract}
This paper continues investigation of the class of flag simple polytopes called 2-truncated cubes. It is an extended version of the short note \cite{V3}. A 2-truncated cube is a polytope obtained from a cube by sequence of truncations of codimension 2 faces. Constructed uniquely defined function which maps any 2-truncated cube to a flag simplicial complex with $f$-vector equal to $\gamma$-vector of the polytope. As a corollary we obtain that $\gamma$-vectors of 2-truncated cubes satisfy Frankl-Furedi-Kalai inequalities. 
\end{abstract}

\section{Introduction}

E.Nevo and T.K.Petersen (see \cite{NP}) studied $\gamma$-vectors of generalized associahedra and proved that $\gamma$-vectors of Stasheff polytopes and Bott-Taubes polytopes can be realized as $f$-vectors of some simplicial complexes.  This result gave rise to the following problem.

\begin{prob*}[cf. \cite{NP}, Problem 6.4]
For given flag simple polytope $P$ construct simplicial complex $\Delta(P)$ such that $\gamma(P) = f(\Delta(P))$.
\end{prob*}

In  \cite{Ai} N.Aisbett solved this problem for flag nestohedra. The construction introduced in \cite{Ai} used specific of building sets and was based on the fact that any flag nestohedron is a 2-truncated cube, i.e. can be obtained from the cube by sequence of truncations of codimension 2 faces (see \cite{V1,V2}).  Results about 2-truncated cubes one can find in \cite{BV}.

In the present paper we introduce the construction which for every 2-truncated cube gives required simplicial complex, i.e. solve the problem for class of all 2-truncated cubes. Moreover, we obtain that constructed complex is flag.
\begin{thm*}
For every 2-truncated cube $P^n$ there exists flag simplicial complex $\Delta(P)$ such that $\gamma(P)=f(\Delta(P))$.
\end{thm*}
In the proof we use the construction that for a given sequence of truncations defines a unique simplicial complex with required $f$-vector. This construction is inductive and build such complexes (on same vertex set) for all faces of the 2-truncated cube. Then we obtain a function $\Delta(Q)$ on the set of faces $G$ of $P$. This function is monotonic, i.e. $\Delta(Q_1)\subset \Delta(Q_2)$ provided by $Q_1\subset Q_2$.
As a corollary we prove that $\gamma$-vectors of 2-truncated cubes satisfy Frankl-Furedi-Kalai inequalities. For dimensions 2 and 3 the required complex is a set of $\gamma_1(P)$ points. For dimensions 4 and 5 the required complex is a graph with $\gamma_1(P)$ vertices and $\gamma_2(P)$ edges without triangles. 

When this paper was in preparation there appeared \cite{Ai2} in Archive. Central result of \cite{Ai2} coinside with the central result of the note \cite{V3} which is a short version of the present paper.

\section{Face polynomials}
The convex $n$-dimensional polytope $P$ is called \emph{simple} if its every vertex belongs to exactly $n$ facets.\medskip\\
Let $f_i$ be the number of $i$-dimensional faces of an $n$-dimensional polytope $P$. The vector $(f_0,\ldots,f_n)$ is called the $f$-vector of $P$. The $F$-polynomial of $P$ is defined by:
\begin{equation*}
F(P)(\alpha,t)=\alpha^n+f_{n-1}\alpha^{n-1}t+\dots +f_1\alpha t^{n-1}+f_0 t^n.
\end{equation*}
The $h$-vector and $H$-polynomial of $P$ are defined by:
\begin{equation*}
H(P)(\alpha,t)=h_0\alpha^n+h_1\alpha^{n-1}t+\dots+h_{n-1}\alpha t^{n-1}+h_n t^n=F(P)(\alpha-t,t).
\end{equation*}
The $g$-vector of a simple polytope $P$ is the vector $(g_0,g_1,\dots,g_{[\frac{n}{2}]})$, where $g_0=1,\quad g_i=h_i-h_{i-1}, i>0$.\medskip\\
The Dehn-Sommerville equations (see \cite{Zi}) state that $H(P)$ is symmetric for any simple polytope. Therefore, it can be represented as a polynomial of $a=\alpha+t$ and $b=\alpha t$:
\begin{equation*}
H(P)=\sum\limits_{i=0}^{[\frac{n}{2}]}\gamma_i(\alpha t)^i(\alpha+t)^{n-2i}.
\end{equation*}
The $\gamma$-vector of $P$ is the vector $(\gamma_0,\gamma_1,\dots,\gamma_{[\frac{n}{2}]})$. The $\gamma$-polynomial of $P$ is defined by:
\begin{equation*}
\gamma(P)(\tau)=\gamma_0+\gamma_1\tau+\dots+\gamma_{[\frac{n}{2}]}\tau^{[\frac{n}{2}]}.
\end{equation*}

\section{Class of 2-truncated cubes}
In this section we introduce the class of 2-truncated cubes. Proofs of the propositions and more results about this class one one can find in \cite{BV}.

\begin{defn}
We say that simple polytope $\tilde P$ is obtained from simple polytope $P$ by truncation of the face $G\subset P$, if simplicial complex $\partial \tilde P^*$ is obtained from the simplitial complex $\partial P^*$ by stellar subdivision along the simplex $\sigma_G$ corresponding to the face $G$. Polytope $\tilde P$ has new facet corresponding to the new vertex $v_0\in \partial \tilde P^*$.
\\~\\
Unformally, polytope $\tilde P$ is obtained from
$P$ by shifting the support hyperplane of $G$ inside polytope $P$. The new facet $\tilde F_s$ of polytope $\tilde P$ corresponding to the new vertex $v_0\in \partial \tilde P^*$ is defined by the section. We will call it the \emph{section facet} $\tilde F_s$.
\end{defn}

\begin{defn}
Truncation of a face $G$ of codimension 2 will be called 2-truncation. A combinatorial polytope obtained from a cube by 2-truncations will be called a 2-truncated cube.
\end{defn}

\begin{rem}\label{new-face}
In this case the section facet will have combinatorial type $G\times I$. After 2-truncation facet $F$ of $P$ either stays unchanged (if $G\supset F$ or $G\cap F=\emptyset$) or handles 2-truncation of face $F\cap G$. Then, for each face $\tilde Q$ of $\tilde P$ there exists a unique face $Q$ such that either $\tilde Q$ is obtained from $Q$ by 2-truncation of $G\cap Q$ or $\tilde Q$ is unchanged (or perturbed) face $Q$ of $P$ or $\tilde Q = Q \times I\subset G\times I$.
\end{rem}

\begin{prop}\label{shave}
Let the $\tilde P$ be obtained from the simple polytope $P$ by 2-truncation of the face $G$, then 
\begin{equation}\label{gamma-change}
\gamma(\tilde P)=\gamma(P)+\tau\gamma(G).
\end{equation}
\end{prop}

\begin{prop}\label{flagshave}
Any 2-truncation keeps flagness.
\end{prop}

\begin{prop}\label{truncated_ring}
Every face of 2-truncated cube is a 2-truncated cube.
\end{prop}

\section{Main results}

Simplicial complex is called \emph{flag}, if its every clique forms a simplex. For simplicial complex $K$ of dimension $d$ the $f$-polynomial is defined by $f(K) := 1 + f_0 t + \dots + f_d t^{d+1}$, where $f_i$ are the numbers of $i$-dimensional faces. The central result of the paper is following.

\begin{thm}
For every 2-truncated cube $P$ there exists a flag  complex $\Delta(P)$ such that $\gamma(P)=f(\Delta(P))$.
\end{thm}

Let $P$ be a 2-truncated cube with fixed sequence of truncations defined by section facets $F_1, \ldots, F_m$. For every face $Q\subset P$ including $Q=P$, let us construct simplicial comlex $\Delta(Q)$ on the vertex set $W(P) = \{w(F_1),\ldots, w(F_m)\}$.

\begin{cons}\label{gamma-complex}
For $P=I^n$ we have $W(P)=\emptyset$ and $\Delta(Q)=\emptyset$ for all the faces.

Assume that required family of simplicial complexes is constructed for polytope $P$ which is obtained from the cube by sequence of 2-truncations corresponding to sequence $F_1,\ldots,F_{m-1}$ of section facets of $P$. Let polytope $\tilde P$ be obtained from $P$ by 2-truncation of face $G_m\subset P$. Then, $W(\tilde P)=W(P)\cup\{w(F_m)\}$, where $w(F_m)$ corresponds to the new facet $F_m$ of $\tilde P$.

 Consider arbitrary face $\tilde Q\subset \tilde P$. Let $Q$ be the face from remark \ref{new-face}. Then,\begin{equation}\label{gamma-complex-defn}
\Delta(\tilde Q):=
\begin{cases}
\Delta(Q) \cup (\Delta(G_m\cap Q)\star w(F_m)),&\text{if $\tilde Q$ is obtained from $Q$ by 2-truncation of  $G_m\cap Q\subset Q$;}\\
\Delta(Q),&\text{otherwise.}\\
\end{cases}
\end{equation}
\end{cons}

\begin{rem}
The number of connected components of $\Delta(P)$ is not greater than number of cubes among truncated faces $G_1,\ldots, G_m$.
\end{rem}

\begin{lem}\label{delta-intersection}
For every $k$-face $Q^k$ of $P$ we have $$\Delta(Q^k) = \bigcap_{F^{n-1}\supset Q^{k}}\Delta(F^{n-1})$$
\end{lem}
\begin{cor}
Function $\Delta(\cdot)$ is monotonic, i.e. $\Delta(Q_1)\subset\Delta(Q_2)$ provided by $Q_1\subset Q_2$.
\end{cor}

\begin{proof}[Proof of lemma \ref{delta-intersection}]
The lemma holds for $P=I^n$. Assume it holds for $P$ and prove it for $\tilde P$ obtained from $P$ by 2-truncation of face $G$.
Notice, that it is enough to prove lemma for faces of codimension 2. Let $\tilde Q=\tilde F_1\cap \tilde F_2\subset \tilde P$ be such a face. According to remark \ref{new-face}, we have 5 possible cases:
\begin{enumerate}
\item Both $\tilde F_1$ and $\tilde F_2$ are facets $F_1$ and $F_2$ of $P$ not changed by truncation;
\item Facet $\tilde F_1$ is obtained from $F_1\subset P$ by 2-truncation, facet $\tilde F_2$ is  unchanged facet  $F_2$ of $P$;
\item Both faces $\tilde F_1$ and $\tilde F_2$ are obtained from faces $F_1$ and $F_2$ of $P$ by 2-truncations;
\item Facet $\tilde F_1$ is the section facet $\tilde F_s$ of $\tilde P$, facet $\tilde F_2$ is unchanged facet  $F_2$ of $P$;
\item Facet $\tilde F_1$ is the section facet $\tilde F_s$ of $\tilde P$, facet $\tilde F_2$  is obtained from $F_2\subset P$ by 2-truncation.
\end{enumerate}
The case 1 is obvious. In the case 2 face $\tilde Q$ is unchanged face $Q$ of $P$. Then, 
\begin{eqnarray*}
\Delta(\tilde F_1)\cap \Delta(\tilde F_2) = (\Delta(F_1) \cup (\Delta(G\cap F_1)\star w(\tilde F_s)))\cap\Delta(F_2)=\\=\Delta(F_1)\cap\Delta(F_2)=\Delta(F_1\cap F_2)=\Delta(\tilde F_1\cap\tilde  F_2) = \Delta(\tilde Q).
\end{eqnarray*}
In the case 3 face $\tilde Q$ is obtained from face $Q=F_1\cap F_2$ by truncation of its face $G \cap Q$. Then,
\begin{eqnarray*}
\Delta(\tilde F_1)\cap \Delta(\tilde F_2)= (\Delta(F_1) \cup (\Delta(G\cap F_1)\star w(\tilde F_s)))\cap (\Delta(F_2) \cup (\Delta(G\cap F_2)\star w(\tilde F_s)))=\\=\Delta(F_1\cap F_2)\cup(\Delta(G\cap F_1 \cap F_2)\star w(\tilde F_s))=\Delta(\tilde F_1\cap \tilde F_2)=\Delta(\tilde Q).
\end{eqnarray*}
In the case 4 we have $\Delta(\tilde F_s)=\Delta(G)\subset \Delta(F_2)$ since $G \subset F_2$. Then,
\begin{eqnarray*}
\Delta(\tilde F_1)\cap \Delta(\tilde F_2) = \Delta(G)\cap \Delta(F_2) = \Delta(G) = \Delta(\tilde F_1\cap \tilde F_2)=\Delta(\tilde Q).
\end{eqnarray*}
In the case 5 we have $\Delta(\tilde F_s\cap \tilde F_2)=\Delta(\tilde F_s\cap \tilde F_2\cap \tilde F_3)$, where $\tilde F_3$ is a facet from the previous case. Then, the required relation follows from the previous cases and from the relation for polytope $\tilde F_3$ which holds by inductive assumption (by dimension).
\end{proof}

\begin{lem}
For every face $Q$ of $P$ complex $\Delta(Q)$ is flag.
\end{lem}
\begin{proof}
On each step of construction \ref{gamma-complex} we merge two flag complexes $\Delta(P_{m-1})$ and $\Delta(G_m)\star w(F_m)$ with flag intersection $\Delta(G_m)$. Then, it is enough to prove that if $\Delta(P)$ contains some edge $\{v_1, v_2\}$ and for some its face $Q$ complex $\Delta(Q)$ contains vertices $v_1$ and $v_2$, then $\Delta(Q)$ contains also edge $\{v_1,v_2\}$.

Without loss of generality we assume that $v_1\in \Delta(G_m)$ and $v_2 = w(F_m)$ . Let $\tilde P$ be obtained from $P$ by 2-truncation of the face $G_m$ and face $\tilde Q$ be obtained from some face $Q$ by 2-truncation of the face $G\cap Q$. We have $v_1\in \Delta(G)$ and $v_1\in \Delta(Q)$, then from lemma \ref{delta-intersection} follows that $v_1\in \Delta(G\cap Q)$. Therefore, the edge $\{v_1, w(F_m)\}$ is contained in $\Delta((G\cap Q)\star w(F_m))\subset \Delta(\tilde Q)$.
\end{proof}

\begin{lem}
For every face $Q$ of $P$ we have $\gamma(Q) = f(\Delta(Q))$.
\end{lem}
\begin{proof}
The lemma holds for $P=I^n$. From the formula \ref{gamma-complex-defn} follows, that if the face $\tilde Q$ is obtained from $Q$ by 2-truncation, then $f(\Delta(\tilde Q))$ and $f(\Delta(Q))$ are connected by the next formula.
$$f(\Delta(\tilde Q)) = f(\Delta(Q) + t f(\Delta (G \cap Q)).$$
Similar formula \eqref{gamma-change} connects $\gamma$-vectors of $\tilde Q$ and $Q$. The lemma follows.
\end{proof}

\begin{thm*}[Frankl-Furedi-Kalai, \cite{FFR}]
Denote by ${\binom n k}_r$ the number of $k$-clique in Turan graph $T_{n,r}$. For natural numbers $m,k$ and $r\geq k$ there exists unique canonical representation $$m = {\binom{n_k}k}_r + \dots + {\binom{n_{k-s}}{k-s}}_{r-s},$$ where $n_{k-i} - [\frac{n_{k-i}}{r-i}]>n_{k-i-1}$ for all $0\leq i < s$ and $n_{k-s}\geq k-s>0$. Denote $$m^{\langle k \rangle_{r}} = {\binom{n_k}{k+1}}_r + \dots + {\binom{n_{k-s}}{k-s+1}}_{r-s}.$$

The integer vector $(f_0,\dots,f_n)$ with nonnegative components is $f$-vector of some $r$-colorable simplicial complex $K$ if and only if $f_k\leq f_{k-1}^{\langle k \rangle_{r}}$.
\end{thm*}

Then, using Frankl-Furedi-Kalai inequalities we obtain the following result which was proved for flag nestohedra in \cite{Ai}.
\begin{cor}
Let $P^n$ be a 2-truncated cube. Then $0\leq \gamma_i\leq \gamma_k^{\langle k\rangle_{r}}$, where $k>1, r=[\frac{n}{2}]$.
\end{cor}

Let us apply the obtained result to polytopes of dimensions 4 and 5. Their $\gamma$-vectors have only 3-components: $(1,\gamma_1,\gamma_2)$. In this case we obtain a graph with $\gamma_1$ vertices and $\gamma_2$ edges without triangles. Therefore, we have 3 inequalities:
\begin{enumerate}
\item $\gamma_1\geq 0$;
\item $\gamma_2\geq 0$;
\item $\gamma_2\leq \frac{\gamma_1(\gamma_1-1)}{2}$.
\end{enumerate}

\small\bigskip
\textsc{Steklov Mathematical Institute,Moscow,Russia}\\
\textsc{Delone Laboratory of Discrete and
Computational Geometry,Yaroslavl State University,Yaroslavl,Russia}\\
\emph{E-mail adress:} \verb"volodinvadim@gmail.com"
\end{document}